\documentclass{amsart}
\usepackage{amsmath,amssymb,amsthm,amsfonts,amsxtra,xspace,mathtools}
\usepackage{enumerate,verbatim,mathrsfs,comment,color,url}
\usepackage[all,2cell,ps]{xy}
\usepackage[pagebackref]{hyperref}
\usepackage{graphicx}
\usepackage{tikz-cd}
\usepackage{verbatim}
\usepackage{stackengine}
\usepackage{cleveref}

\newcommand{\mycomment}[1]{}

\newcommand{\new}[1]{{\color{blue} #1}}

\DeclareMathOperator{\ann}{ann}
\DeclareMathOperator{\Ass}{Ass}

\DeclareMathOperator{\depth}{depth}

\DeclareMathOperator{\Ext}{Ext}
\DeclareMathOperator{\fgmod}{mod}

\DeclareMathOperator{\grade}{grade}
\DeclareMathOperator{\Hom}{Hom}

\DeclareMathOperator{\id}{id}
\DeclareMathOperator{\im}{im}

\DeclareMathOperator{\pd}{pd}
\DeclareMathOperator{\Proj}{Proj}

\DeclareMathOperator{\Supp}{Supp}
\DeclareMathOperator{\Tor}{Tor}

\renewcommand{\ge}{\geqslant}
\renewcommand{\le}{\leqslant}

\newcommand{\ds}{\displaystyle}

\newcommand{\bn}{\mathbb{N}}
\newcommand{\bz}{\mathbb{Z}}

\newcommand{\fm}{\mathfrak{m}}

\newcommand{\sm}{\mathscr{M}}

\newcommand{\xlra}{\xlongrightarrow}

\newcommand{\ud}{\underline{d}}
\newcommand{\un}{\underline{n}}
\newcommand{\um}{\underline{m}}
\newcommand{\uc}{\underline{c}}

\theoremstyle{plain}
\newtheorem{theorem}{Theorem}[section]
\newtheorem{lemma}[theorem]{Lemma}
\newtheorem{proposition}[theorem]{Proposition}

\theoremstyle{definition}
\newtheorem{definition}[theorem]{Definition}

\newtheorem{example}[theorem]{Example}

\newtheorem{notations}[theorem]{Notations}
\newtheorem{notation}[theorem]{Notation}

\newtheorem{question}[theorem]{Question}
\newtheorem{setup}[theorem]{Setup}

\theoremstyle{remark}
\newtheorem{remark}[theorem]{Remark}

\numberwithin{equation}{section}

\title[Coherent functors, powers of ideals, and asymptotic stability]{Coherent functors, powers of ideals, and asymptotic stability}

\author[S.~Dey]{Souvik Dey}
\address{Department of Algebra, Faculty of Mathematics and Physics, Charles University in Prague, Sokolovsk\'{a} 83, 186 75 Praha, Czech Republic}
\email{souvik.dey@matfyz.cuni.cz} 
\urladdr{\url{https://orcid.org/0000-0001-8265-3301}}

\author[D.~Ghosh]{Dipankar Ghosh}
\address{Department of Mathematics, Indian Institute of Technology Kharagpur, West Bengal - 721302, India}
\email{dipankar@maths.iitkgp.ac.in, dipug23@gmail.com}
\urladdr{\url{https://orcid.org/0000-0002-3773-4003}}

\author[S.~Pramanik]{Siddhartha Pramanik}
\address{Department of Mathematics, Indian Institute of Technology Kharagpur, West Bengal - 721302, India}
\email{siddharthap@kgpian.iitkgp.ac.in, pramaniksiddhartha2@gmail.com}

\author[T.J.~Puthenpurakal]{Tony J. Puthenpurakal}
\address{Department of Mathematics, Indian Institute of Technology Bombay, Powai, Mumbai 400076, India}
\email{tputhen@math.iitb.ac.in}

\author[S.~Sahoo]{Samarendra Sahoo}
\address{Department of Mathematics, Indian Institute of Technology Bombay, Powai, Mumbai 400076, India}
\email{204093008@iitb.ac.in}

\date{\today}
\subjclass[2020]{Primary 13D07, 13A15; Secondary 13A02, 13D02}
\keywords{Coherent functors; Graded rings and modules; Associate primes; Hilbert functions, Betti and Bass numbers}

\begin{document}
\pagenumbering{arabic}
\thispagestyle{empty}
    
\begin{abstract}
Let $R$ be a Noetherian ring, $I_1,\ldots,I_r$ be ideals of $R$, and $N\subseteq M$ be finitely generated $R$-modules. Let $S = \bigoplus_{\un \in \bn^r} S_{\un}$ be a Noetherian standard $\bn^r$-graded ring with $S_{\underline{0}} = R$, and $\sm $ be a finitely generated $\bz^r$-graded $S$-module. For $ \un = (n_1,\dots,n_r) \in \bn^r$, set $G_{\un} := \sm_{\un}$ or $G_{\un} := M/{\bf I}^{\un} N$, where ${\bf I}^{\un} = I_1^{n_1} \cdots I_r^{n_r}$.
Suppose $F$ is a coherent functor on the category of finitely generated $R$-modules. We prove that the set $\Ass_R \big(F(G_{\un}) \big)$ of associate primes and $\grade\big(J, F(G_{\un})\big)$ stabilize for all $\un \gg 0$, where $J$ is a non-zero ideal of $R$. Furthermore, if the length $\lambda_R(F(G_{\un}))$ is finite for all $\un \gg 0$, then there exists a polynomial $P$ in $r$ variables over $\mathbb{Q}$ such that $\lambda_R(F(G_{\un})) = P(\un)$ for all $\un\gg 0$. When $R$ is a local ring, and $G_{\un} = M/{\bf I}^{\un} N$, we give a sharp upper bound of the total degree of $P$.
As applications, when $R$ is a local ring, we show that for each fixed $i \ge 0$, the $i$th Betti number $\beta_i^R(F(G_{\un}))$ and Bass number $\mu^i_R(F(G_{\un}))$ are given by polynomials in $\un$ for all $\un \gg 0$. Thus, in particular, the projective dimension $\pd_R(F(G_{\un}))$ (resp., injective dimension $\id_R(F(G_{\un}))$) is constant for all $\un\gg 0$.
\end{abstract}
\maketitle

\section{Introduction}
Let $R$ be a (commutative) Noetherian ring, and let $\fgmod(R)$ denote the category of finitely generated $R$-modules.
Let $X$ be a nonempty set, and $\{ x_{\lambda} \}_{\lambda \in \Lambda}$ be a sequence of elements in $X$ indexed by a partially ordered set $\Lambda$. We say that the sequence $\{ x_{\lambda} \}_{\lambda \in \Lambda}$ `becomes stable for all large $\lambda$' or it `stabilizes asymptotically' if there exists some $\lambda_0 \in \Lambda$ such that $x_{\lambda} = x_{\lambda_0}$ for all $\lambda \ge \lambda_0$. In \cite{B79}, Brodmann paved the path for studying asymptotic stability of associate primes by proving the sets $\Ass_R(I^n M/I^{n+1} M)$ and $\Ass_R(M/I^nM)$ stabilize for all $n\gg 0$, where $M \in \fgmod(R)$ and $I$ is an ideal of $R$. In \cite{Brodmann79}, Brodmann proved that if $J$ is an ideal of $R$, then both $\grade(J, I^n M/ I^{n+1} M)$ and $\grade(J, M/ I^{n} M)$ stabilize asymptotically. Over the years, many generalizations of these results have happened in different forms. Also, these results of Brodmann inspired researchers to study the asymptotic behaviour of other algebraic invariants. Here, we present a few findings that motivated us to write this article.

In \cite[Thm.~1]{MS93}, Melkersson-Schenzel proved that for each fixed $i \ge 0$, the sets $\Ass_R(\Tor_i^R(M,R/I^n))$ and $\Ass_R(\Tor_i^R(M,I^n/I^{n+1}))$ become independent of $n$ for all $n \gg 0$. In \cite[Thm.~1.5]{KS96}, Kingsbury-Sharp showed that if $I_1, \dots, I_r$ are ideals of $R$, and $N , N' \in \fgmod(R)$ with $N' \subseteq N$, then the set $\Ass_R(N/I_1^{n_1}\cdots I_r^{n_r} N')$ becomes stable for all large $(n_1,\ldots,n_r)$. Later, in \cite[Thm.~3.5.(i) and Cor.~3.10.(i)]{West04}, West recovered the result of Kingsbury-Sharp in a direct way, and also showed that if $J$ is an ideal of $R$, then $\grade(J, N/I_1^{n_1}\cdots I_r^{n_r}N)$ stabilizes asymptotically. Extending the result on associate primes, Katz-West in \cite[Cor.~3.5]{KW04} proved that for each fixed $i \ge 0$, the sets $\Ass_R(\Ext_R^i(M, N/I_1^{n_1}\cdots I_r^{n_r}N'))$ and $\Ass_R(\Tor_i^R(M, N/I_1^{n_1}\cdots I_r^{n_r}N'))$ become stable for all large $(n_1,\ldots,n_r)$. See \cite{GP19} for a concise proof of the result of Katz-West.

In \cite[Thm.~2]{Ko93}, Kodiyalam showed that if the length $\lambda_R(M\otimes_RN)$ is finite, then for every fixed $i\ge 0$, the functions $\lambda_R(\Ext_R^i(M, N/I^nN))$ and $\lambda_R(\Tor_i^R(M, N/I^nN))$ both become polynomials in $n$ for all $n\gg 0$. Theodorescu, in \cite[Cor.~4]{Th02}, proved the same results with a weaker condition, and
gave sharp bounds of the degree of these polynomials.

Researchers also gave a far-reaching generalization of Brodmann's results in the context of coherent functors. Before we come to those results, we recall the definition of coherent functors. For $M \in \fgmod(R)$, denote the covariant functor $h_M(-) : = \Hom_R(M,-)$.

\begin{definition}(\cite[p.~189]{A66} and \cite[p.~53]{H98})
    A covariant $R$-linear functor F on $\fgmod(R)$ is said to be coherent if there exists an exact sequence $ h_K \to h_M \to F \to 0$ of functors for some finitely generated $R$-modules $K$ and $M$.
\end{definition}

In the setting of Abelian categories, Auslander introduced coherent functors in \cite{A66}. Later, Hartshorne in \cite{H98} studied coherent functors on the category $\fgmod(R)$ in detail. The following examples of coherent functors are known in the literature.

\begin{example}
Let $M \in \fgmod(R)$, and ${\bf C_{\bullet}}$ be a chain complex of finitely generated $R$-modules.
\begin{enumerate}[\rm (1)]
    \item \cite[Exam.~2.3.(a)]{H98}
    For each $i \in \bz$, the $i$th homology functor $H_i({\bf C_{\bullet}} \otimes_{R}-)$ is coherent.
    \item {\cite[Exam.~2.4 and 2.5]{H98}}
    For each $i \ge 0$, the functors $\Tor_i^R(M,-)$ and $\Ext_R^i(M,-)$ are coherent.
    \item \cite[Cor.~3.3 and Rmk.~1.10]{T17}
    Let $I$ be an ideal of $R$. Then the $0$th local cohomology functor $H^0_I(-)$ with support in $I$ is coherent if and only if $I^n = I^{n+1}$ for some $ n \ge 0$.
    \item \cite[Def.~3.9 and Cor.~3.12]{T17}
    Let $S$ be a common multiplicatively closed subset of $R$. Then the torsion functor $\tau_S(-)$ is coherent if and only if $S$ is coprincipal.
    \item \cite[Thm.~3.4]{BM19}
    If $F$ and $G$ are two coherent functors on $\fgmod(R)$, then so is the composition $G\circ F$.
    \end{enumerate}
\end{example}

For a coherent functor $F$ on $\fgmod(R)$, Tony Se in \cite[Thm.~1.11]{T17} proved that the sets of associate primes $\Ass_R(F(I^n M/ I^{n+1} M))$ and $\Ass_R(F(M/ I^n M))$ and the numerical sequences $\grade(J, F(I^n M/ I^{n+1} M))$ and $\grade(J, F(M/ I^n M))$ stabilize asymptotically. In \cite[Thms.~2.5 and 3.5]{BM19}, Banda-Melkersson gave simple proofs of these results, mainly by obtaining some properties of coherent functors. Moreover, they showed that if $\lambda_R(F(M/ I^n M))$ is finite for all $n$, then the function $\lambda_R (F(M/ I^n M))$ coincides with a polynomial in $n$ for all $n \gg 0$, cf.~\cite[Thm.~3.3]{BM19}.
    
Keeping all the results mentioned above in mind, we give a common generalization. In this context, we mainly prove the following theorem.

\begin{theorem}[See \Cref{th:main2}]
    Let $R$ be a Noetherian ring, $I_1, \dots, I_r, J$ be ideals of $R$, and $N \subseteq M$ be finitely generated $R$-modules. Suppose $F$ is a coherent functor on the category of finitely generated $R$-modules. Then both $\Ass_R(F(M/I_1^{n_1} \cdots I_r^{n_r} N))$ and $\grade(J, F(M/I_1^{n_1} \cdots I_r^{n_r} N))$ eventually stabilize. Moreover, if $F(M/I_1^{n_1} \cdots I_r^{n_r} N)$ has finite length for all $n_i \gg 0$, then $\lambda_R(F(M/I_1^{n_1} \cdots I_r^{n_r} N))$ is eventually given by a polynomial $P$ over $\mathbb{Q}$ in $(n_1,\ldots,n_r)$. Furthermore, if $R$ is a local ring, then there is a sharp upper bound of the degree of $P$ given by $\max\{ \dim(F(M)), \ell_M({\bf I}) - r\}$, {\rm cf.~\Cref{notation-deg}}.
\end{theorem}

Assume that $R$ is a local ring. Fix $i \ge 0$. In \cite[Cor.~7]{Ko93}, Kodiyalam proved that the $i$th Betti number $\beta_i^R(M/I^n M)$ and the $i$th Bass number $\mu_R^i(M/I^n M)$ are polynomials in $n$ for all sufficiently large $n$. He also provided an upper bound on the degree of these polynomials. More generally, in \cite[Thm.~3.7]{BM19}, Banda-Melkersson showed that $\beta_i^R(F(M/I^n M))$ and $\mu_R^i(F(M/I^n M))$ are polynomials in $n$ for all $n \gg 0$. As applications of our main theorems, we prove that $\beta_i^R(F(M/I_1^{n_1} \cdots I_r^{n_r} N))$ and $\mu_R^i(F(M/I_1^{n_1} \cdots I_r^{n_r} N))$ are eventually given by polynomials in $(n_1,\ldots, n_r)$. Moreover, there is an upper bound of the degrees of these polynomials given by $\max\{ 0, \ell_M({\bf I}) - r \} $, which is independent of $i$, $F$ and $N$. Furthermore, from these results, we deduce the asymptotic stability of certain homological invariants (namely, projective dimension and injective dimension) in our context, see \Cref{thm:appl}.

The rest of the article is arranged as follows. 
In \Cref{sec3}, we first observe some stability results for the modules, which are obtained by applying a coherent functor to the components of a finitely generated multigraded module over a Noetherian standard multigraded ring. Then, we prove a sequel of results to provide the proof of \Cref{th:main2}. \Cref{sec4} contains some applications of our results.

\section{Main results}\label{sec3}
In this section, we mainly prove \Cref{th:main2}. First, we prepare some results that provide the ingredients for the proof of \Cref{th:main2}. To present our results better, we fix a few notations. 

\begin{notations}\label{notations}
    Let $\bn$ denote the set of non-negative integers, and $\bz$ denote the set of integers. Let $r\ge 1$ be an integer. A typical element of $\bz^r$ is denoted by $\un:=(n_1, \dots, n_r)$. Set $\underline{0} := (0,\ldots,0)$. 
    For $\um, \un\in\bz^r$, define $\um \ge \un$ if $m_i \ge n_i$ for all $1\le i\le r$. By `for all $\un \gg 0$', we mean `for all $\un \ge \uc$ for some $\uc \in \bn^r$'. For $c \in \bz$ and $\un \in \bz^r$, define $c\un := (cn_1, \dots, cn_r)$. 
\end{notations}

Throughout the article, we work with the following setup.

\begin{setup}\label{setup}
    Unless specified, let $R$ be a Noetherian ring. Denote the category of finitely generated $R$-modules by $\fgmod(R)$. Let $I_1,\dots,I_r$ be ideals of $R$. Denote ${\bf I}^{\un} := I_1^{n_1} \cdots I_r^{n_r}$, where $\underline{n} =(n_1, \dots, n_r) \in \bn^r$.
\end{setup} 

We start by noting the following observation.

\begin{theorem}\label{th:main1}
    Let $S = \bigoplus_{\un \in \bn^r}S_{\un}$ be a Noetherian standard $\bn^r$-graded ring with $S_{\underline 0} = R$. Suppose $F$ is a coherent functor on $\fgmod(R)$. Let $\sm = \bigoplus_{\un \in \bz^r}{\sm}_{\un}$ be a finitely generated $\bz^r$-graded $S$-module. Then, $F(\sm) := \bigoplus_{\un \in \bz^r}F({\sm}_{\un})$ is also a finitely generated $\bz^r$-graded $S$-module. Furthermore:
    \begin{enumerate}[\rm (1)]    
    \item 
    The set $\Ass_R\big(F({\sm}_{\un})\big)$ stabilizes asymptotically.    
    \item 
    For any non-zero ideal $J$ of $R$, the value $\grade \big(J,F({\sm}_{\un})\big)$ stabilizes asymptotically.
    \item \label{th:main:poly1} 
    If $\lambda_R \big(F(\sm_{\un})\big)$ is finite for all $\un \gg 0$, then there exists a polynomial $P$ in $r$ variables over $\mathbb{Q}$ such that $\lambda_R \big(F(\sm_{\un})\big) = P(\un)$ for all $\un\gg 0$.
    \end{enumerate}
\end{theorem}

\begin{proof}
Since $F$ is coherent, there exist $L$ and $K$ in $\fgmod(R)$ such that $ h_L \to h_K \to F \to 0$ is an exact sequence of functors. Thus, $\bigoplus_{\un \in \bz^r}F(\sm_{\un})$ is a quotient of the finitely generated $\bz^r$-graded $S$-module $\bigoplus_{\un \in \bz^r}h_K(\sm_{\un})$. Hence $\bigoplus_{\un \in \bz^r}F(\sm_{\un})$ is also a finitely generated $\bz^r$-graded $S$-module.
Consequently, the first two statements follow from \cite[Thm.~3.4.(i)]{West04} and \cite[Cor.~3.9.(i)]{West04} respectively. In order to prove the third statement, assume that $\lambda_R(F(\sm_{\un}))$ is finite for all $\un \gg 0$. Then, by (1), for all $\un \gg 0$, the set $\Ass_R(F(\sm_{\un}))$ consists of finitely many maximal ideals of $R$, say $\fm_1, \dots, \fm_t$. Then, for all $\un \gg 0$, one has that $ \lambda_R(F(\sm_{\un})) = \sum_{i=1}^t\lambda_{R_{\fm_i}}\big(F((\sm_{\un})_{\fm_i})\big)$. Thus, without loss of generality, we may assume that $R$ is a local ring. Then the desired result is a consequence of \cite[Thm.~3.2]{F02}.
\end{proof}

We record the following homological lemma for future use.

\begin{lemma}\label{lem1}
Let $X,Y,Z,A,B,C$ and $F$ be $R$-modules that fit into the following commutative diagram of $R$-linear maps, where the rows and columns are exact.
$$\begin{tikzcd}
Z                                       & C                     &           & \\
Y \ar[u, "\gamma"]  \ar[r, "\alpha"]    & B \ar[u, "\beta"]     &           & \\
X \ar[u, "\eta"]    \ar[r, "\theta"]    & A \ar[r] \ar[u, "j"]  & F \ar[r]  & 0\\
                                        & 0 \ar[u]
\end{tikzcd}$$
Then, $F\cong {\ker(\beta)}/{\alpha(\ker \gamma)}$.
\end{lemma}

\begin{proof}
Since $j$ is injective, it follows that $A/\im(\theta)\cong j(A)/\im(j\circ \theta)$. Since $j\circ \theta=\alpha\circ \eta$, one has $\im(j\circ \theta)=\im(\alpha\circ \eta)=\alpha(\operatorname{\im} \eta)=\alpha(\ker \gamma)$, where the last equality is due to the exactness of the first column. Thus,  $A/\im(\theta) \cong j(A)/\alpha(\ker\gamma) = \ker (\beta)/\alpha(\ker\gamma)$, where the last equality follows from the exactness of the second column. Since the last row is exact, $F \cong {A}/{\im(\theta)}$. This finishes the proof of the claim.  
\end{proof}

The following lemma is a consequence of the multigraded Artin-Rees lemma. The idea of this lemma first appeared in the proof of \cite[Prop.~3]{Th02}. The single ideal case of this lemma is shown in \cite[Cor.~2.2]{T17}. We include the proof because the construction of $U$ in the lemma is utilized to prove our main results.

\begin{lemma}\label{lem:homology}
    With {\rm \Cref{setup}}, let $A \xlra{\phi}B \xlra{\psi} C$ be a complex of modules in $\fgmod(R)$.
    Let $A'\subseteq A$, $B' \subseteq B$ and $ C' \subseteq C$ be submodules such that $\phi(A') \subseteq B' $ and $ \psi(B') \subseteq C' $. Let $\uc \in \bn^r$, and $A_1$, $A_2$ be submodules of $A$ such that ${\bf I}^{\uc}A' \subseteq A_2$. For $\un \ge \uc$, let $H(\un)$ denote the homology of the induced complex
    \begin{equation}\label{eq:complex}
        \frac{A_1 + {\bf I}^{\un - \uc}A_2}{{\bf I}^{\un}A'} \xlra{\phi(\un)} \frac{B}{{\bf I}^{\un}B'} \xlra{\psi(\un)} \frac{C}{{\bf I}^{\un}C'}.
    \end{equation}    
    Then, there exist $\ud \in \bn^r$ such that $H(\un) \cong (U+{\bf I}^{\un - \ud}V)/{\bf I}^{\un - \ud}W$ for all $\un \ge \ud $, where $U,V,W$ are some submodules of a finitely generated $R$-module $T$ satisfying $W \subseteq V$ and $U \cong \ker(\psi)/ \phi(A_1)$.
\end{lemma}

\begin{proof}
    By the multigraded Artin-Rees lemma (cf.~\cite[17.1.6]{SH06}), there exists $\ud \in \bn^r$ such that $ \ud \ge \uc$ and
    \[
    \psi(B) \cap {\bf I}^{\un} C' = {\bf I}^{\un - \ud} \big( \psi(B) \cap {\bf I}^{\ud} C'\big) \; \mbox{ for all } \un \ge \ud.
    \]
    It follows that $\psi^{-1} ({\bf I}^{\un}C') = \ker(\psi) + {\bf I}^{\un - \ud}\big(\psi^{-1}({\bf I}^{\ud}C')\big)$ for all $\un \ge \ud$.
    Hence, for $\un \ge \ud$,
    \begin{align*}
        H(\un) = \frac{\ker(\psi(\un))}{\im(\phi(\un))} &= \frac{\psi^{-1} ({\bf I}^{\un}C')/ {\bf I}^{\un}B'}{[\phi(A_1 + {\bf I}^{\un - \uc}A_2) + {\bf I}^{\un}B']/{\bf I}^{\un}B'} \\
        &\cong \frac{\ker(\psi) + {\bf I}^{\un - \ud}\big(\psi^{-1}({\bf I}^{\ud}C')\big)}{\phi(A_1) + {\bf I}^{\un - \ud} \big({\bf I}^{\ud - \uc}\phi(A_2) + {\bf I}^{\ud}B'\big)} \cong \frac{U + {\bf I}^{\un - \ud}V}{{\bf I}^{\un - \ud}W},
    \end{align*}
    where
    \begin{equation}\label{UVW-description}
        U := \frac{\ker(\psi)}{\phi(A_1)}, \quad V := \frac{\psi^{-1}({\bf I}^{\ud}C') + \phi(A_1)}{\phi(A_1)} \quad \mbox{and} \quad W := \frac{{\bf I}^{\ud - \uc}\phi(A_2) + {\bf I}^{\ud}B' + \phi(A_1)}{\phi(A_1)}.
    \end{equation}
    Note that $U,V$ and $W$ all are $R$-submodules of $T := B/\phi(A_1)$. Moreover, $W \subseteq V$.
\end{proof}

In proving our next theorem, we employ techniques from \cite[Proof of Cor.~2.4]{T17}. This theorem serves as the primary tool in the proof of \Cref{th:main2}.

\begin{theorem}\label{thm:interpret}
    With {\rm \Cref{setup}}, let $F$ be a coherent functor on $\fgmod(R)$. Suppose $M, N \in \fgmod(R)$ with $N \subseteq M$. Then, there exist $\ud \in \bn^r$ and some submodules $U$, $V$ and $W$ of a finitely generated $R$-module $T$ such that $U \cong F(M)$, $W \subseteq V$, and
    \[
    F\left(\frac{M}{{\bf I}^{\un}N}\right) \cong \frac{U + {\bf I}^{\un - \ud}V}{{\bf I}^{\un - \ud}W} \; \mbox{ for all } \un \ge \ud.
    \]
\end{theorem}

\begin{proof}
    Since $F$ is coherent, there exist $L$ and $K$ in $\fgmod(R)$ such that $ h_L \to h_K \to F \to 0$ is an exact sequence of functors. Then, by Yoneda lemma (cf.~\cite[Thm.~1.17]{Rotman09}), the map $ h_L \to h_K$ is induced by an $R$-module homomorphism $f:K \to L$. By choosing free presentations of $K$ and $L$, and lifting $f:K \to L$ to these presentations, one obtains the following commutative diagram.
    \begin{equation} \label{eqn:lift}
    \begin{split}
    \xymatrix{
      R^{\oplus k_1} \ar[r] \ar[d]^{\beta}      & R^{\oplus \ell_1} \ar[d]^{\gamma}\\
      R^{\oplus k_0} \ar[r]^{\alpha} \ar[d]     & R^{\oplus \ell_0} \ar[d]\\
      K \ar[r]^f \ar[d]                         & L \ar[d] \\
      0                                         & 0.
    }
    \end{split}
    \end{equation}
  By applying the functor $\Hom_R(-,M/{\bf I}^{\un}N)$ to \eqref{eqn:lift}, one gets the following commutative diagram.
  $$
    \xymatrix{
    \ds \frac{M^{\oplus \ell_1}}{{\bf I}^{\un} \left(N^{\oplus \ell_1}\right)}
    \ar[r] &
    \ds \frac{M^{\oplus k_1}}{{\bf I}^{\un} \left(N^{\oplus k_1}\right)}\\
    \ds \frac{M^{\oplus \ell_0}}{{\bf I}^{\un} \left(N^{\oplus \ell_0}\right)}
    \ar[u]_{\gamma_{\un}^{*}} \ar[r]^{\alpha_{\un}^{*}} &
    \ds \frac{M^{\oplus k_0}}{{\bf I}^{\un} \left(N^{\oplus k_0}\right)}
    \ar[u]_{\beta_{\un}^{*}}\\
    h_L\left( \ds \frac{M}{{\bf I}^{\un} N} \right)
    \ar[u] \ar[r]^{f_{\un}^{*}} &
    h_K\left( \ds \frac{M}{{\bf I}^{\un} N} \right)
    \ar[r] \ar[u] &
    F \left( \ds \frac{M}{{\bf I}^{\un} N} \right)
    \ar[r] & 0\\
    0 \ar[u] &  0 \ar[u]
    } 
  $$
  where the maps $f_{\un}^{*},\alpha_{\un}^{*},\beta_{\un}^{*},\gamma_{\un}^{*}$ are induced by $f,\alpha,\beta,\gamma$ respectively. Then, \Cref{lem1} yields that
\begin{equation}\label{F-M-mod_InN}
    F\left(\frac{M}{{\bf I}^{\un} N}\right)
    \cong \frac{\ker \beta_{\un}^{*}}
          {\alpha_{\un}^{*} \big( \ker \gamma_{\un}^{*} \big)} \, .
\end{equation}
Similarly, applying the functor $\Hom_R(-,M)$ to the diagram \eqref{eqn:lift}, one obtains the maps $\alpha^{*},\beta^{*},\gamma^{*}$ induced by $\alpha,\beta,\gamma$ respectively, and the isomorphism $F(M) \cong \ker(\beta^*) / \alpha^*(\ker \gamma^*)$. Set
$A := M^{\oplus \ell_0}$, $A' := N^{\oplus \ell_0}$ and  $D' := N^{\oplus \ell_1}$. We show that there exist some $R$-submodules $A_1$ and $A_2$ of $A$, and $\uc \in \bn^r$ satisfying ${\bf I}^{\uc} A' \subseteq A_2$ and $\ker(\gamma_{\un}^*) = (A_1 + {\bf I}^{\un -\uc}A_2)/{\bf I}^{\un}A'$ for all $\un \ge \uc$.
Note that both $\gamma^*(A)$ and $D'$ are submodules of $M^{\oplus \ell_1}$. So, by the multigraded Artin-Rees lemma (\cite[17.1.6]{SH06}), there exists $\uc \in \bn^r$ such that
\[
  \gamma^*(A) \cap {\bf I}^{\un}D' = {\bf I}^{\un - \uc} \big( \gamma^*(A) \cap {\bf I}^{\uc} D' \big) \; \mbox{ for all } \un \ge \uc.
\]
It follows that
\[
  (\gamma^*)^{-1} \big( {\bf I}^{\un} D' \big) = \ker(\gamma^*) + {\bf I}^{\un - \uc}\big((\gamma^*)^{-1}({\bf I}^{\uc} D')\big) \; \mbox{ for all } \un \ge \uc.
\]
Set $A_1 := \ker(\gamma^*)$ and $A_2 := (\gamma^*)^{-1}({\bf I}^{\uc} D')$. Then ${\bf I}^{\uc} A' \subseteq A_2$. Moreover, for all $\un \ge \uc$, one has that
\[
  \ker(\gamma_{\un}^*) =
  \frac{(\gamma^*)^{-1}\big({\bf I}^{\un} D'\big)}{{\bf I}^{\un} A'} =
  \frac{\ker(\gamma^*) + {\bf I}^{\un - \uc}\big((\gamma^*)^{-1}({\bf I}^{\uc} D')\big)}{{\bf I}^{\un} A'} =
  \frac{A_1 + {\bf I}^{\un -\uc} A_2}{{\bf I}^{\un} A'}.
\]
Consequently, in view of \eqref{F-M-mod_InN}, $F\big({M}/{{\bf I}^{\un} N}\big)$ is the homology module of the complex 
\begin{equation}\label{eq:complex2}
    \frac{A_1 + {\bf I}^{\un -\uc} A_2}{{\bf I}^{\un} A'} \xlra{\alpha_{\un}^*|_{\ker\gamma_{\un}^*}}
    \frac{M^{\oplus k_0}}{{\bf I}^{\un} \left(N^{\oplus k_0}\right)} \xlra{\beta_{\un}^*} 
    \frac{M^{\oplus k_1}}{{\bf I}^{\un} \left(N^{\oplus k_1}\right)},
\end{equation}
which has the form \eqref{eq:complex}, where $B := M^{\oplus k_0}$, $B' := N^{\oplus k_0}$, $C := M^{\oplus k_1}$, $C' := N^{\oplus k_1}$, $\phi(\un) := {\alpha_{\un}^{*}|_{\ker \gamma_{\un}^*}}$ and $\psi(\un) := \beta_{\un}^*$. The last two maps are induced by $\phi := \alpha^{*}|_{\ker \gamma^*}$ and $\psi := \beta^*$ respectively. Therefore, by \Cref{lem:homology}, there exist $\ud \in \bn^r$ and some submodules $U$, $V$ and $W$ of a finitely generated $R$-module $T$ such that $ W \subseteq V$ and $F(M/{\bf I}^{\un} N) \cong (U + {\bf I}^{\un - \ud}V)/{\bf I}^{\un - \ud}W$ for all $\un \ge \ud$.
Furthermore,
\[
U \cong \frac{\ker \psi} {\phi (A_1)} = \frac{\ker \beta^*} {\alpha^*(A_1)} = \frac {\ker \beta^*} {\alpha^*( \ker \gamma^*)} \cong F(M).
\]
This completes the proof.
\end{proof}

The following result might be known to the experts, see, e.g., \cite[Thm.~17.4.2]{SH06}, where it is assumed that the base ring $(R,\fm)$ is local and the ideals $I_1,\ldots,I_r$ are $\fm$-primary. We note that these two conditions can be removed. Due to the lack of references, we add its proof here.

\begin{proposition}\label{prop:hspoly}
    With {\rm \Cref{setup}}, let $M \in \fgmod(R)$. Suppose $\lambda_R(M/{\bf I}^{\un}M)$ is finite for all $\un \gg 0$. Then there exists a polynomial $P$ in $r$ variables over $\mathbb{Q}$ of total degree $\dim_R(M)$ such that
    $$
    \lambda_R(M/{\bf I}^{\un}M) = P(\un) \; \mbox{ for all } \un \gg 0.
    $$
\end{proposition}

\begin{proof}
    Since $\lambda_R(M/{\bf I}^{\un}M)$ is finite for all $\un \gg 0$, by \cite[Thm.~1.5]{KS96}, it follows that there exist maximal ideals $\fm_1, \dots, \fm_t$ of $R$ such that $\Ass_R(M/{\bf I}^{\un}M) = \{ \fm_1, \dots, \fm_t \}$ for all $\un \gg 0$. Therefore
    $$
    \lambda_R(M/{\bf I}^{\un}M) = \sum_{j=1}^t{\lambda_{R_{\fm_j}} \big(M_{\fm_j}/{\bf I}^{\un} M_{\fm_j} \big)} \; \mbox{ for all } \un \gg 0.
    $$
    Hence, localizing at each maximal ideal $\fm_j$, and replacing $R_{\fm_j}$ by $R$, we may assume that the base ring $(R,\fm)$ is local. If $R' := R/ \ann_R(M)$, then $\lambda_R(M/{\bf I}^{\un}M) = \lambda_{R'}(M/{\bf I}^{\un}M)$. So there is no harm in assuming $\ann_R(M) = 0$. It follows that
    $$
    \{\fm\} = \Supp_R(M/{\bf I}^{\un}M) = \Supp_R(M \otimes_R R/{\bf I}^{\un}) = \Supp_R(M) \cap V({\bf I}^{\un}) = V({\bf I}^{\un}) =\bigcup_{i=1}^r V(I_i).
    $$
    Thus, the ideals $I_1,\ldots,I_r$ all are $\fm$-primary, and the desired result follows from \cite[Thm.~17.4.2]{SH06}.
\end{proof}

\begin{notation}\label{notation-deg}
    With \Cref{setup}, further assume that $R$ is a local ring with the maximal ideal $\fm$. For $M \in \fgmod(R)$, let $\ell_M({\bf I})$ denote the (Krull) dimension of $\bigoplus_{\un \in \bn^r}{{\bf I}^{\un} M/ \fm {\bf I}^{\un} M}$ over the multigraded Rees ring $\mathcal{R}({\bf I}):= \bigoplus_{\un \in \bn^r}{{\bf I}^{\un}}$. By degree of a multivariable polynomial, we mean the total degree.
\end{notation}

In the following theorem, for a $\bz^r$-graded module $M$ over an $\bn^r$-graded ring $S = \bigoplus_{\un \in \bn^r} S_{\un}$, let $\Proj(S)$ denote the set of all homogeneous prime ideals of $S$ which do not contain the ideal $\bigoplus_{\un \ge \underline{1}} S_{\un}$, where $\underline{1} = (1,\ldots,1)$, and let $ \Supp_{++}(M) := \Supp(M) \cap \Proj(S)$. The single ideal case of the following theorem is shown in \cite[Lem.~2]{Th02}.  

\begin{theorem}\label{th:poly}
    With {\rm \Cref{setup}}, let $U$, $V$ and $W$ be finitely generated submodules of an $R$-module $T$ with $W \subseteq V$. Set $L_{\un} := (U +{\bf I}^{\un} V)/{\bf I}^{\un} W$ for all $ \un \in \bn^r$. Assume that $\lambda_R(L_{\un})$ is finite for all $ \un \gg 0$. Then, there exists a polynomial $P$ in $r$ variables over $\mathbb{Q}$ such that
    \[
    \lambda_R(L_{\un}) = P(\un) \; \text{ for all } \un \gg 0.
    \]
    Furthermore, if $R$ is a local ring, and $J \subseteq \ann_R(V)$ is an ideal of $R$, then following {\rm \Cref{notation-deg}},
    \[
    \deg(P) \le \max\{\dim(U), \ell_{R/J}({\bf I})-r \}.
    \]
\end{theorem}

\begin{proof} 
    For $ \un \in \bn^r$, consider the following short exact sequence:
    \begin{equation}\label{eq:ses}
        0 \to \frac{(U + {\bf I}^{\un} W) \cap {\bf I}^{\un} V} {{\bf I}^{\un} W} \longrightarrow 
        \frac{U + {\bf I}^{\un} W} {{\bf I}^{\un} W} \bigoplus \frac{{\bf I}^{\un} V}{{\bf I}^{\un} W} \longrightarrow
        L_{\un} \to 0.
    \end{equation}
    All three modules in \eqref{eq:ses} have finite length for all $\un \gg 0$. So, it suffices to prove that the lengths of the first two modules in \eqref{eq:ses} are given by polynomials in $\un$ for all $\un \gg 0$. Since $ \bigoplus_{\un \in \bn^r}{ {\bf I}^{\un} V}/{{\bf I}^{\un} W} $ is a finitely generated $\bn^r$-graded module over the $\bn^r$-graded Rees ring $\mathcal{R}({\bf I})$, by \cite[Thm.~4.1]{HHRT97}, $\lambda_R({\bf I}^{\un} V/{{\bf I}^{\un} W)}$ is given by a polynomial in $\un$ for all $\un \gg 0$. Since $ (U + {\bf I}^{\un} W) \cap {\bf I}^{\un} V \subseteq {\bf I}^{\un} V$, the $\bn^r$-graded module 
    \[
    \bigoplus_{\un \in \bn^r}  {\frac{(U + {\bf I}^{\un} W) \cap {\bf I}^{\un} V} {{\bf I}^{\un} W}}
    \]
    is also finitely generated over $\mathcal{R}({\bf I})$, and hence similar conclusion holds for its graded components. It remains to consider the function $\lambda_R\big((U+{\bf I}^{\un}W)/{{\bf I}^{\un} W\big)}$. By the multigraded Artin-Rees lemma, there exists $\ud \in \bn^r$ such that 
    \[
    \frac{U + {\bf I}^{\un} W} {{\bf I}^{\un} W} \cong \frac{U} {U \cap {\bf I}^{\un} W} = \frac{U} {{\bf I}^{\un - \ud} (U \cap {\bf I}^{\ud} W)} \; \mbox{ for all } \un \ge \ud.
    \]
    Therefore, for all $\un \ge \ud$, one has that
    \begin{equation}\label{eq:length}
    \lambda_R \left( \frac{U + {\bf I}^{\un} W} {{\bf I}^{\un} W} \right) =
    \lambda_R \left( \frac{U}{U \cap {\bf I}^{\ud} W} \right) +
    \lambda_R \left( \frac{U \cap {\bf I}^{\ud} W} {{\bf I}^{\un - \ud} (U \cap {\bf I}^{\ud} W)} \right).
    \end{equation}
    In view of \Cref{prop:hspoly}, the last term of \eqref{eq:length} is given by a polynomial in $\un$ for all $\un \gg 0$.
    Consequently, $\lambda_R\big((U+{\bf I}^{\un}W)/{{\bf I}^{\un} W\big)}$ has a polynomial behaviour for all $ \un \gg 0$.

    For the second part, assume that $R$ is a local ring with the maximal ideal $\fm$. Consider an ideal $J \subseteq \ann_R(V)$. Set $\bar{(-)} := (-) \otimes_R R/J$. From the short exact sequence \eqref{eq:ses}, one obtains that
    \begin{equation}\label{deg-T-bound}
    \deg \lambda_R(L_{\un}) \le 
    \max\left\{ \deg \lambda_R \big((U + {\bf I}^{\un} W)/ {\bf I}^{\un} W \big), \deg \lambda_R \big({\bf I}^{\un} V/ {\bf I}^{\un} W \big)\right\}.
    \end{equation}
    Note that the $\bn^r$-graded module $\bigoplus_{\un \in \bn^r}{ {\bf I}^{\un} V}/{{\bf I}^{\un} W} $ is finitely generated over the ring $\bar{\mathcal{R}}(\bar{{\bf I}})$. In view of \cite[Lem.~3.4.(ii)]{Gh16}, there is some $\uc \in \bn^r$ such that $\ann_R({\bf I}^{\un} V/{{\bf I}^{\un} W})$ stabilizes for all $\un \ge \uc$. Set $\mathcal{H} := \bigoplus_{\un \ge \uc}{ {\bf I}^{\un} V}/{{\bf I}^{\un} W} $. Since $\mathcal{H}_{\un}$ has finite length for all $\un \gg 0$, the stabilized ideal $J_0 := \ann_R(\mathcal{H}_{\un})$ for $\un \ge \uc$ is $\fm$-primary. Note that $\mathcal{H}$ is a finitely generated $\bn^r$-graded module over the $\bn^r$-graded ring $\mathcal{S} := \bar{\mathcal{R}}(\bar{{\bf I}})/ \bar{J_0} \bar{\mathcal{R}}(\bar{{\bf I}})$. Hence, in view of \cite[Thm.~4.1~and~Lem.~1.1]{HHRT97},
    \begin{equation}\label{deg-H-bound}
    \deg \lambda_R (\mathcal{H}_{\un}) =
    \dim\big(\Supp_{++}(\mathcal{H})\big) \le \dim\big( \Proj(\mathcal{S}) \big) \le 
    \dim(\mathcal{S}) - r.
    \end{equation}
    Since $J_0$ is $\fm$-primary, it follows that
    \begin{equation}\label{dim-S}
    \dim(\mathcal{S}) =
    \dim\big(\bar{\mathcal{R}}(\bar{{\bf I}})/ \bar{\fm} \bar{\mathcal{R}}(\bar{{\bf I}})\big) = 
    \dim \Big(\bigoplus_{\un \in \bn^r} {\bf I}^{\un} \bar{R}/  \fm{\bf I}^{\un}\bar{R} \Big) = \ell_{R/J}({\bf I}).
    \end{equation}
    Now, in view of \eqref{eq:length}, if $U \cap {\bf I}^{\ud} W = 0$, then $\lambda_R(U) < \infty$ and $ \deg \lambda_R \big( (U + {\bf I}^{\un} W)/ {\bf I}^{\un} W \big) = 0 = \dim (U)$. So, we may assume that $U \cap {\bf I}^{\ud} W \neq 0$. Since $\lambda_R(U/ (U \cap {\bf I}^{\ud} W)) < \infty$, one gets that $ \Supp(U \cap {\bf I}^{\ud} W) = \Supp(U)$, and hence $\dim(U \cap {\bf I}^{\ud} W) = \dim(U) $. Therefore, the equalities in \eqref{eq:length} yield that
    \begin{equation}\label{deg-U-mod-W}
    \deg \lambda_R \left( \frac{U + {\bf I}^{\un} W} {{\bf I}^{\un} W} \right) =
    \deg  \lambda_R \left( \frac{U \cap {\bf I}^{\ud} W} {{\bf I}^{\un - \ud} (U \cap {\bf I}^{\ud} W)} \right) = \dim(U \cap {\bf I}^{\ud} W) = \dim(U).
    \end{equation}
    The desired bound of $\deg(P)$ follows from \eqref{deg-T-bound}, \eqref{deg-H-bound}, \eqref{dim-S} and \eqref{deg-U-mod-W}.
\end{proof}

Combining Theorems~\ref{thm:interpret} and \ref{th:poly}, the eventual polynomial behaviour of $\lambda_R(F(M/{\bf I}^{\un}N))$ is obtained provided that $\lambda_R(F(M/{\bf I}^{\un}N))$ is finite for all $\un\gg 0$. Next, we show a sharp upper bound of the degree of this polynomial which depends only on $F$, $M$ and $I_1,\ldots,I_r$.

\begin{theorem}\label{th:poly2}
    With {\rm \Cref{setup}}, let $R$ be a local ring, and $F$ be a coherent functor on $\fgmod(R)$. Suppose $M, N \in \fgmod(R)$ with $N \subseteq M$. Assume that $\lambda_R(F(M/{\bf I}^{\un}N))$ is finite for all $\un \gg 0$. Then, following {\rm \Cref{notation-deg}},
    \begin{equation}\label{eqn:deg-ineq}
        \deg \lambda_R(F(M/{\bf I}^{\un}N)) \le \max \{ \dim(F(M)),\, \ell_M({\bf I}) - r \},
    \end{equation}
    where the inequality becomes equality if $ \dim(F(M)) > \ell_M({\bf I}) -r$. 
\end{theorem}

\begin{proof}
    Set $J := \ann_R(M)$. We use the notations as described in \Cref{thm:interpret}. So $\dim(U) = \dim(F(M))$. In view of the proof of \Cref{thm:interpret}, note that $J$ annihilates every module in the complex \eqref{eq:complex2}. So $J$ annihilates every subquotient of the modules in \eqref{eq:complex2}. Hence, using the description of $V$ given in the proof of \Cref{lem:homology}, one gets that $J \subseteq \ann_R(V)$. Thus, by Theorems~\ref{thm:interpret} and \ref{th:poly}, in order to establish the inequality \eqref{eqn:deg-ineq}, it is enough to show that $ \ell_{R/J}({\bf I}) = \ell_{M}({\bf I}) $. Note that $ \ell_M({\bf I}) $ is the (Krull) dimension of $\mathcal{M} := \bigoplus_{\un \in \bn^r}{{\bf I}^{\un} M/ \fm {\bf I}^{\un} M}$ over the ring $\mathcal{R}({\bf I})$, hence over the ring $\bigoplus_{\un \in \bn^r}{{\bf I}^{\un} (R/J)/ \fm {\bf I}^{\un} (R/J)}$. Since 
    \[
    \bigoplus_{\un \in \bn^r}\frac{{\bf I}^{\un} (R/J)} {\fm {\bf I}^{\un} (R/J)} \cong \bigoplus_{\un \in \bn^r}\frac{{\bf I}^{\un} + J} {\fm {\bf I}^{\un} + J} \cong
    \bigoplus_{\un \in \bn^r} \frac{{\bf I}^{\un}} {({\bf I}^{\un} \cap J) + \fm {\bf I}^{\un}},
    \]
    it follows that $\bigoplus_{\un \in \bn^r}\big(({\bf I}^{\un} \cap J) + \fm {\bf I}^{\un}\big) \subseteq \ann_{\mathcal{R}({\bf I})}(\mathcal{M})$. Now it is enough to prove that $\ann_{\mathcal{R}({\bf I})}(\mathcal{M})$ is equal to $\bigoplus_{\un \in \bn^r}(({\bf I}^{\un} \cap J) + \fm {\bf I}^{\un})$ up to radical. As $\ann_{\mathcal{R}({\bf I})}(\mathcal{M})$ is a homogeneous ideal of $ \mathcal{R}({\bf I})$, consider a homogeneous element $x \in {\bf I}^{\um} \cap \ann_{\mathcal{R}({\bf I})}(\mathcal{M})$. Then $x M \subseteq \fm {\bf I}^{\um} M$. Hence, by the determinant trick, $x \in \overline{ \fm{\bf I}^{\um} }$ modulo $J$, where the bar stands for the integral closure of $\fm{\bf I}^{\um}$ in $R$. So there exist $c \in \bn$, $a_i \in (\fm {\bf I}^{\um})^i$ for $1\le i\le c$, and $ y \in J$ such that
    \[
    x^{c} + a_1 x^{c - 1} + \cdots + a_{c-1} x + a_c = y.
    \]
    Thus $ x^c \in ({\bf I}^{c\um} \cap J)+ \fm {\bf I}^{c\um}$. This proves our claim.

    For the last part, suppose $ \dim(F(M)) > \ell_M({\bf I}) - r$. Note that $U \cong F(M)$. In view  of the short exact sequence \eqref{eq:ses}, it is sufficient to prove that
    $$ \deg \lambda_R \left( \frac{{\bf I}^{\un} V}{{\bf I}^{\un} W} \right) < \deg \lambda_R \left( \frac{U + {\bf I}^{\un} W} {{\bf I}^{\un} W} \right) .$$
    From the proof of \Cref{th:poly}, one observes that $\deg \lambda_R ( {\bf I}^{\un} V /{\bf I}^{\un} W)$ is bounded above by $\ell_{R/J}({\bf I}) -r$. Since $\ell_{R/J}({\bf I}) -r = 
    \ell_{M}({\bf I}) - r <\dim(U) = \deg \lambda_R \big( (U + {\bf I}^{\un} W)/ {\bf I}^{\un} W \big)$, the proof is complete.
\end{proof}

\begin{remark}
    \Cref{th:poly2} highly generalizes the result \cite[Cor.~4]{Th02} of Theodorescu.
\end{remark}

Now, we are in a position to state and prove the main results of this article.

\begin{theorem}\label{th:main2}
    With {\rm \Cref{setup}}, let $F$ be a coherent functor on $\fgmod(R)$. Suppose $M, N \in \fgmod(R)$ with $N \subseteq M$. Then, the following hold.
    \begin{enumerate}[\rm (1)]    
    \item 
    The set $\Ass_R \big(F(M/{\bf I}^{\un}N) \big)$ stabilizes asymptotically. 
    \item 
    For a non-zero ideal $J$ of $R$, the numeric value $\grade \big(J,F(M/{\bf I}^{\un}N)\big)$ stabilizes asymptotically.
    \item \label{th:main:poly2} 
    If $\lambda_R \big(F(M/{\bf I}^{\un}N)\big)$ is finite for all $\un \gg 0$, then there exists a polynomial $P$ in $r$ variables over $\mathbb{Q}$ such that $\lambda_R \big(F(M/{\bf I}^{\un}N) \big) = P(\un)$ for all $\un\gg 0$. In addition, if $R$ is a local ring, with {\rm \Cref{notation-deg}},
    \[
        \deg(P) \le \max\{ \dim(F(M)), \ell_M({\bf I}) -r \},
    \]
    where the inequality becomes equality if $\dim(F(M)) > \ell_M({\bf I}) -r $.
    \end{enumerate}
\end{theorem}

\begin{proof}
(1) In view of \Cref{thm:interpret} there exists $\ud \in \bn^r$ such that
\[
F\left(\frac{M}{{\bf I}^{\un}N}\right) \cong \frac{U + {\bf I}^{\un - \ud}V}{{\bf I}^{\un - \ud}W}
\]
for some submodules $U$, $V$ and $W$ of a finitely generated $R$-module $T$ with $W \subseteq V$. By \cite[Prop.~3.4]{KW04} and its proof, $\Ass_R \big((U + {\bf I}^{\un}V)/{\bf I}^{\un}W \big)$ stabilizes for all $\un \gg 0$. Consequently, (1) follows.

(2) Set $L_{\un} := (U + {\bf I}^{\un}V)/{\bf I}^{\un}W$ for all $ \un \in \bn^r$. Consider a non-zero ideal $J$ of $R$. Then
\[
\frac{L_{\un}}{JL_{\un}} \cong \frac{\overline{U} + {\bf I}^{\un}\overline{V}}{{\bf I}^{\un} \overline{W}},
\]
where $\overline{U} := U/JU$, $\overline{V} := (V+JU)/JU$ and $\overline{W} := (W+JV+JU)/JU$ are submodules of $\overline{T} := T/JU$ with $\overline{W} \subseteq \overline{V}$. By \cite[Prop.~3.4]{KW04},  $\Ass_R(L_{\un}/JL_{\un})$ stabilizes for all $\un \gg 0$. Hence, if $L_{\un} = JL_{\un}$ for all $\un \gg 0$, then $\grade \big(J,F(M/{\bf I}^{\un}N)\big) = \infty $ for all $\un \gg 0$. So we may assume that $L_{\un} \neq JL_{\un}$ for all $ \un \gg 0$. Suppose $J = (a_1 , \dots, a_s)$. Then $0\le \grade \big(J,F(M/{\bf I}^{\un}N)\big) \le s$ for all $\un \gg 0$. Since the composition of two coherent functors is coherent (cf.~\cite[Thm.~3.4]{BM19}), the functor $\Ext_R^i \big(R/J,F(-) \big)$ is coherent. Therefore, by (1), the set $\Ass_R \big( \Ext_R^i \big(R/J,F(M/{\bf I}^{\un}N) \big) \big)$ stabilizes for all $\un \gg 0$. Denote $X_i := \Ass_R \big( \Ext_R^i \big(R/J,F(M/{\bf I}^{\un}N) \big) \big)$ for all large $\un$. Then there exists $l$ with $0\le l \le s$ such that $X_l \neq \emptyset$ and $X_i = \emptyset$ for all $0\le i <l$. Hence, since
\[
\grade \big(J,F(M/{\bf I}^{\un}N)\big) = \min \big\{ i : \Ext_R^i \big(R/J,F(M/{\bf I}^{\un}N) \big) \neq 0 \big\},
\]
it follows that $\grade \big(J,F(M/{\bf I}^{\un}N)\big) = l$ for all large $\un$.

(3) This simply follows by combining Theorems~\ref{thm:interpret}, \ref{th:poly} and \ref{th:poly2}.
\end{proof}

\section{Applications}\label{sec4}

As an application of our main theorems, we prove the following result. Note that \Cref{thm:appl}.(1) considerably strengthens both \cite[Cor.~7]{Ko93} and \cite[Thm.~3.7]{BM19} in many directions.

\begin{theorem}\label{thm:appl}
    With {\rm \Cref{setup}}, let $R$ be a local ring. Suppose $M,N \in \fgmod(R)$ with $N \subseteq M$. Let $S = \bigoplus_{\un \in \bn^r} S_{\un}$ be a Noetherian standard $\bn^r$-graded ring with $S_{\underline{0}} = R$, and $\sm $ be a finitely generated $\bz^r$-graded $S$-module. Suppose $F$ is a coherent functor on $\fgmod(R)$. Denote $G_{\un} := \sm_{\un}$ or $G_{\un} := M/{\bf I}^{\un} N$ for all $ \un \in \bn^r$. Then, the following hold.
    \begin{enumerate}[\rm (1)]
    \item 
    For each fixed $i \ge 0$, the functions $\beta_i^R(F(G_{\un}))$ and $\mu^i_R(F(G_{\un}))$ are eventually given by polynomials in $\un$. When $G_{\un} := M/{\bf I}^{\un} N$, the degrees of these polynomials are bounded above by $\max \{ 0, \ell_M({\bf I}) - r \} $.
    \item 
    Both projective dimension $\pd_R(F(G_{\un}))$ and injective dimension $\id_R(F(G_{\un}))$ are eventually constants.
    \end{enumerate}
\end{theorem}

\begin{proof}
     Assume $k$ is the residue field of $R$. Since the composition of two coherent functors is coherent, for each fixed $i \ge 0$, the functors $\Tor_i^R(k,F(-))$ and $\Ext^i_R(k,F(-))$ are also coherent. Consequently, in view of \Cref{th:main1}.\eqref{th:main:poly1} and \Cref{th:main2}.\eqref{th:main:poly2}, both the functions $\beta_i^R(F(G_{\un})) = \lambda_R(\Tor_i^R(k,F(G_{\un})))$ and $\mu^i_R(F(G_{\un})) = \lambda_R(\Ext_R^i(k,F(G_{\un})))$ are given by polynomials in $\un$. Since $\Tor_i^R(k,F(M))$ and $\Ext_R^i(k,F(M))$ are zero-dimensional $R$-modules, when $G_{\un} := M/{\bf I}^{\un} N$, \Cref{th:main2}.\eqref{th:main:poly2} ensures that the degrees of the polynomials are bounded above by $\max\{ 0, \ell_M({\bf I}) - r \} $. This proves (1). In order to prove (2), set $d:=\depth(R)$. If $\beta_{d+1}^R(F(G_{\un}))$ is eventually given by a non-zero polynomial in $\un$, then $\pd_R(F(G_{\un})) \ge d+1$ for all $\un\gg 0$, and hence by the Auslander-Buchsbaum formula, $\pd_R(F(G_{\un})) = \infty$ for all $\un \gg 0$. On the other hand, if $\beta_{d+1}^R(F(G_{\un}))$ is eventually a zero polynomial in $\un$, then $\pd_R(F(G_{\un})) \le d$ for all $\un\gg 0$. In that case, suppose $t$ is the largest $i$ such that $\beta_{i}^R(F(G_{\un}))$ is eventually a non-zero polynomial in $\un$. Then $t \le d$, and $\pd_R(F(G_{\un})) = t$ for all $\un \gg 0$. For every non-zero finitely generated $R$-module $L$, note that $\id_R(L)=d$ or $\id_R(L)=\infty$ depending on whether $\mu^{d+1}_R(L)=0$ or $\mu^{d+1}_R(L)\neq 0$ respectively. So, a similar argument using the eventual polynomial behaviour of $\mu^i_R(F(G_{\un}))$ gives the asymptotic stability of $\id_R(F(G_{\un}))$.
\end{proof} 


\section*{Acknowledgments}
 Dey was partly supported by the Charles University Research Center program No.
UNCE/24/SCI/022 and a grant GACR 23-05148S from the Czech Science Foundation. Pramanik and Sahoo thank the Government of India for financial support through the Prime Minister Research Fellowship (PMRF) and the University Grants Commission (UGC) scholarship, respectively, for their Ph.D.

\bibliographystyle{plain}
\bibliography{mainbib}

\end{document}